\documentclass[10pt]{article}
\usepackage{amsmath, amsthm}
\usepackage{amssymb}
\usepackage[hypertex]{hyperref}

\theoremstyle{theorem}
\newtheorem{theorem}{Theorem}[section]
\newtheorem{proposition}[theorem]{Proposition}
\newtheorem{lemma}[theorem]{Lemma}

\theoremstyle{definition}
\newtheorem{definition}{Definition}[section]

\theoremstyle{remark}
\newtheorem{remark}{Remark}[section]

\theoremstyle{example}
\newtheorem{example}{Example}[section]

\theoremstyle{notation}
\newtheorem{notation}{Notation}[section]

\newcommand{\al}{\alpha}
\newcommand{\be}{\beta}

\newcommand{\la}{\lambda}

\newcommand{\Ann}{{\rm Ann}}

\newcommand{\gm}{\mathfrak{m}}

\title{The resultants of quadratic binomial complete intersections}
\author{T.\  Harima, Niigata University \\ Department of Mathematics Education, Niigata, 950-2181 Japan
\thanks{Supported by JSPS KAKENHI Grant 15K04812.} 
\and A.\ Wachi, Hokkaido University of Education \\  Department of Mathematics,  
\\ Kushiro, 085-8580 Japan 
\and  J.\ Watanabe, Tokai University \\ Department of Mathematics, Hiratsuka, 259-1292 Japan
}
\begin{document}

\maketitle
\date{}

\def\pa{{\partial}}
\begin{abstract} 
We compute the resultants for quadratic binomial complete intersections. As an 
application we show that any quadratic binomial complete intersection can have the set
of square-free monomials as a vector space basis if the generators are put
in a normal form.
\end{abstract}

\section{Introduction}
Consider the homogeneous polynomial in 5 variables of degree 5
\[F=12vwxyz+ t_1+t_2, \]
where 
\[t_1=-2(p_1v^3yz + p_2vw^3z + p_3vwx^3 + p_4wxy^3 + p_5xyz^3),\]
\[t_2=p_1p_3v^3x^2 + p_2p_4w^3y^2 + p_3p_5x^3z^2 + p_1p_4v^2y^3 + p_2p_5w^2z^3.\]

The polynomial $F$, as the Macaulay dual generator of an Artinian Gorenstein algebra in the Macaulay's inverse system,
defines a flat family of  Gorenstein algebras with
Hilbert function $(1 \ 5 \ 10 \ 10 \ 5 \ 1)$. If
$p_1p_2p_3p_4p_5 + 1 \neq 0$, then the ideal which $F$ defines  is
$5$ generated, and if  $p_1p_2p_3p_4p_5 + 1 =  0$, it is 7 generated. All the members in this
family possess the strong Lefscehtz property. This is known from the computation of
the 1st and the 2nd Hessians of the form $F$.  
On the other hand there exists a similar, but slightly more complicated  form  $G$ (see \S6),
involving $5$ parameters $p_1, \ldots, p_5$, which  
has generic  Hilbert function $(1 \ 5 \ 10 \ 10 \ 5 \ 1)$, but if $p_1p_2p_3p_4 p_5+1=0$ then
the Hilbert function  reduces to $(1 \ 5 \ 5 \ 5 \ 5 \ 1)$.  It seems remarkable that the second
Hessian of $G$ is divisible by $(p_1p_2p_3p_4 p_5+1)^5$.  A striking fact is that it is precisely equal to
the resultant of the complete intersection as the generic member of this family. 

In this paper we do not pursue the Hessians; instead, we study the resultant for
quadratic complete intersections.
Generally speaking the resultant of $n$ forms in $n$ variables is a very
complicated polynomial in the coefficients of the forms.  The meaning of the resultant is
that it does not vanish if and only if the ideal generated by the $n$ forms
is a complete intersection.
Thus if the $n$ forms are monomials, the resultant is easy to  describe.

Through the investigation of the strong Lefschetz property of complete intersections,
we found it necessary to acquire the skill to
calculate  the  resultant of complete intersections.  The theory of the resultants
has a long history, but a  method to obtain it is described completely
in the book \cite{GKZ}. Generic computation is impossible except in 
a few cases because of the huge number of variables involved and the high degree of the resultant.
Hence, if one wishes to deal with the resultant in an arbitrary number of variables, it is
reasonable to restrict the  attention to a special class of homogeneous polynomials.  

The purpose of this paper is to describe the resultant of quadratic binomials 
 in $n$ variables.  We restricted our attention only to quadratic forms. 
 It is because, though it seems needless to say, the degree-two condition  makes the computation simple. 
 However,  in addition to it, we can 
 expect that a lot of information can be obtained from the knowledge of 
 quadratic complete intersections. Indeed the authors~\cite{hww_1} 
 showed  that a certain family of
 complete intersections is obtained as subrings of quadratic complete intersections.
 In addition,  McDaniel~\cite{cris_mcdaniel}    showed that many complete intersections could be embedded in quadratic
 complete intersections. 
 We can expect, from the manner their  results were proved, a vast
 amount of complete intersections to be  obtained as subrings of quadratic complete intersections
 sharing the socle and general linear elements.

Suppose that $I=(f_1, \ldots, f_n)$ is an ideal generated by  $n$ homogeneous 
polynomials of the same degree $\la$ in the polynomial ring $R$.  
Then $I$ is a complete intersection if and only if $I_{nd-n+1}=R_{n\la -n- d +1}I_d $ is the whole homogeneous 
space $R_{n\la - n  + 1}$.  With the aid of  an idea of Gelfand et.\ al~\cite{GKZ}, it is possible  to pick up 
certain polynomials from among the elements in  $I_{\la n-n+1}$ such that their span is the whole space.
We apply their method to the particular situation 
where the generators are quadratic binomials.  
It turns out that the same method can be used
to obtain the conditions which guarantees that, for smaller values of $\lambda$,
the elements in  $I_{\lambda}$ generate the subspace of
$R_{\lambda}$ complementary to the space spanned by the  square-free monomials. 
Thus our computation unexpectedly 
gives us a proof that the complete intersections defined by quadratic binomials can have
square free monomials as a vector space basis. We state it as Theorem~\ref{main_theorem}. 
This follows from Theorem~\ref{from_gelfand's_book}, which is the main result of this paper.   

This paper is organized as follows. In \S2 we show that any $n$-dimensional quadratic 
space can have a normal form. This is an elementary observation but it seems that
it has never been used systematically before.  
It reduces the amount of the
computation of the resultant considerably even in the generic case. 
In \S3 we investigate some properties of the determinants of square matrices of the  ``binomial type'',  
which we need  
in the sequel. In \S4, we introduce matrices consisting of rows as coefficients of  the polynomials in the ideal $I_{\la}$.
Then we  apply the result of \S2 to finding the determinants of these matrices.   
In \S5 we give some examples of quinternary quintics and 
in \S6 we briefly discuss the  relation between the 2nd Hessian of the Macaulay dual and the
resultant of the quadratic binomials. We refer the reader to \cite[Definition~3.75]{HMMNWW} or \cite{maeno_watanabe}
for the definition of the 2nd Hessian. 

\section{The normal form of a quadratic vector space} 

\begin{definition} 
  A binomial is a polynomial of the form  $f=\al M + \be N$, where $M,N$ are distinct monomials in certain variables
  and $\al, \be$ are some elements in a field.   
\end{definition}

\begin{notation} \normalfont
  We denote by  $R=K[x_1, x_2, \ldots, x_n]$  the polynomial ring in $n$ variables over a field $K$. 
  We denote by  $R_{\lambda}$  the homogeneous space of degree $\lambda$ of $R$.  
\end{notation}

\begin{proposition} \normalfont  \label{quadratic_space_in_normal_form} 
Let $V \subset R_2$ be an $n$-dimensional vector  subspace of $R_2$. 
Then by a linear transformation of the variables we may choose a basis
$f_1, f_2, \ldots, f_n$ for $V$ such that
\[f_i=x_i^2 + \fbox{\mbox{\rm linear combination of square-free monomials}}, i=1,2, \ldots, n.\]
\end{proposition}

\begin{proof}
  We induct on $n$. If $n=1$, then the assertion is trivial. 
  Assume that $n > 1$.
  Suppose that $V$ is spanned by
  \[f_1, f_2, \ldots, f_n.\]
  Let $M=(m_{ij})$ be the matrix defined by
\[m_{ij} = \fbox{ \mbox{the coefficient of $x_j ^2$ in  $f_i$} }.\] 
  By a linear transformation of variables if necessary, we may assume that 
  $f_1, \ldots, f_{n-1}$ are linearly independent by reduction $x_n=0$.  
  Then by the induction hypothesis we may assume that $M$ takes the form:
  
\[M=
\left(
\begin{array}{ccccc}
  1&0&0&0&*\\
  0&1&0&0&*\\
   & &\ddots& & \\
  0&0&\cdots &1&*\\
  *&*&\cdots &*&*
\end{array}  
\right). 
\]

First assume  $\mbox{det}\;M \neq 0$. Define
\[\begin{pmatrix} g_1\\g_2\\ \vdots \\ g_n\end{pmatrix} = M^{-1}\begin{pmatrix} f_1\\f_2\\ \vdots \\ f_n 
\end{pmatrix}.\]
Then $g_1, \ldots, g_n$ are a  desired basis for $V$.  Next assume that
$\mbox{det}\;M = 0$. This means that we may make the last row of $M$ a zero vector after subtracting a linear
combination of $f_1, \ldots, f_{n-1}$ from $f_n$. 
Consider the linear transformation of the variables
\begin{align*}
  {}& x_i \mapsto x_i+ \xi _i x_n, \mbox{ for } i=1, 2, \ldots, n-1, \\
  {}& x_n \mapsto x_n.
\end{align*}
With this transformation only the last column of $M$ is affected and non-zero
element appears as the $(n,n)$ entry. (This is because $f_n$ contains some square-free monomial with non-zero coefficient.)
Thus we are in the situation as $\mbox{det }M\neq 0$.  
\end{proof}

\begin{definition}
Suppose that $f_1, \ldots, f_n \in R_2$ are linearly independent. We will say that
these elements are in a {\bf normal form} as a basis for a quadratic vector space if
\[f_i=x_i^2 + \fbox{\mbox{\rm linear combination of square-free monomials}}, i=1,2, \ldots, n.\]
\end{definition}

\begin{example}　\normalfont
  Consider  $R=K[x,y,z]$, $V=\langle x^2, y^2, xy \rangle$.  
  Then we have the matrix expression:
  \[\begin{pmatrix}  x^2\\y^2\\xy  \end{pmatrix}=
  \begin{pmatrix}  1&0&0&0&0&0 \\ 0&1&0&0&0&0 \\  0&0&0&1&0&0 \end{pmatrix}
  \begin{pmatrix}  x^2\\y^2\\z^2\\xy\\xz\\yz \end{pmatrix}
  \]
  \newcommand{\GL}{{\rm GL}}
  Make the translation $\sigma \in \GL(3)$ : $x\mapsto x + \xi z$,  $y \mapsto y + \eta z$ and $z \mapsto z$. 
  Then we have the isomorphism of vector spaces $V \cong \langle f,g,h   \rangle$, where 
  \[\begin{pmatrix}  f\\g\\h \end{pmatrix}=
  \begin{pmatrix}  1&0&\xi ^2&0&2\xi&0 \\ 0&1&\eta^2&0&0&2 \eta \\  0&0&\xi \eta &1&\eta & \xi \end{pmatrix}
  \begin{pmatrix}  x^2\\y^2\\z^2\\xy\\xz\\yz \end{pmatrix}.
  \]
  Furthermore we have the equality of vector spaces:
  $\langle  f,g ,h  \rangle = \langle  f',g' ,h'  \rangle$, 
  where
\[\begin{pmatrix}  f'\\g'\\h' \end{pmatrix}=
\begin{pmatrix}  1&0&\frac{-\xi}{\eta} \\ 0&1&\frac{-\eta}{\xi} \\ 0&0&\frac{1}{\xi\eta}\end{pmatrix}
  \begin{pmatrix}  f\\g\\h \end{pmatrix}.
  \]
  So the vector space  $V$ is transformed to  $\langle  f',g',h' \rangle$, where 
  \[\begin{pmatrix}  f'\\g'\\h' \end{pmatrix}=
  \begin{pmatrix}
    1&0&0&\frac{-\xi}{\eta}&\xi&\frac{-\xi ^2}{\eta} \\
    0&1&0&\frac{-\eta}{\xi}&\frac{-\eta ^2}{\xi}& \eta  \\
    0&0&1&\frac{1}{\xi\eta } &\frac{1}{\xi}&\frac{1}{\eta}  \end{pmatrix}
  \begin{pmatrix}  x^2\\y^2\\z^2\\xy\\xz\\yz \end{pmatrix}.
  \]
  It follows that  we have the isomorphism of the algebras $K[x,y,z]/(V)\cong K[x,y,z]/(f'g',h')$ with 
  the elements $f', g', h'$ are put in a normal form.
  
  \end{example}

For a later purpose we prove some propositions in linear algebra.

\begin{proposition}  \label{binomial_matrix} \normalfont 
  Let 
  $A:=\{a_1, a_2, \ldots, a_N\}$ and $B:=\{b_1, b_2, \ldots, b_N\}$ be two independent sets of variables.
Suppose that $P$ is an $N \times N$ square matrix satisfying the following conditions:
  \begin{enumerate}
    \item
      The $i$th row of $P$ contains  $a_i$ and $b_i$ as entries and  all the other entries are zero.
    \item
      Any column of $P$ contains an element in $A$ and an element in $B$.  
  \end{enumerate}
  Then the following conditions are equivalent.
  \begin{enumerate}
  \item[(1)] $\det P$ is an irreducible polynomial as an element  in the polynomial ring
    \[R:=K[a_1, \ldots, a_N, b_1, \ldots, b_N].\]
  \item[(2)] The matrix $P$ is irreducible, i.e., does not split to blocks
    like $\begin{pmatrix}P_1&O\\O&P_2 \end{pmatrix}$ by permutation of rows and columns.
     \item[(3)] $\det P=\pm(a_1a_2\cdots a_N +(-1)^{N+1} b_1b_2\cdots b_N)$.  
    \end{enumerate}
\end{proposition}

\begin{proof}
  Assume that $P$ splits into blocks. Then obviously $\det P$ is reducible. This proves (1) implies (2).
  Assume that $P$ is irreducible.  Permuting the columns do not affect the condition of $P$.  So we
  may assume that the diagonal entries of $P$ are $(a_1, a_2, \ldots, a_N)$.
  Furthermore we may conjugate $P$ by a permutation matrix so that 
  $a_1, \ldots, a_N$ are diagonal entries and elements of $B$ are distributed in the super-diagonal entries
  and at the $(N,1)$-position. 
  This enables us to compute the determinant of $P$  as (3). Thus we have
  proved that (2) implies (3).  It remains to show that (3) implies (1).  In other words we have to show that
  $d:=a_1 \cdots a_N \pm b_1 \cdots b_N$ is an irreducible polynomial in $R$. 
  It is easy to see that we have the isomorphism 
  \[K[a_1, \ldots,a_N, b_1, \ldots, b_N]/(d, b_1-1, b_2-1, \ldots, b_N -1)\cong K[a_1, \ldots, a_N]/(a_1\ldots a_N \pm 1),\]
  and  that this algebra is an integral domain of Krull dimension $N-1$.  This shows that
  \[d, b_1-1, \ldots, b_N-1 \] is a regular sequence in $R$ and furthermore $R/(d)$ is an integral domain. 
\end{proof}

In the next proposition we slightly weaken the conditions on $P$ and prove similar properties.  

\begin{proposition}  \label{binomial_matrix_2} \normalfont 
  Let  $A:=\{a_1, a_2, \ldots, a_N\}$ and $B:=\{b_1, b_2, \ldots, b_N\}$ be independent sets of variables.
  Suppose that $P$ is an $N \times N$ square matrix which satisfies the following conditions. 
  \begin{enumerate}
    \item
      The $i$th row of $M$ contains $a_i$ and $b_i$ as entries and  all the other entries are zero.
    \item
      Any column of $M$ does not contain two elements in   $A$. 
  \end{enumerate}
  Then we have: 
  \begin{enumerate}
     \item[(1)] The variable $a_i$ divides $\det P$ if $a_i$ is the only non-zero element in the column that contains $a_i$.   
     \item[(2)] $\det P$ is independent of $b_i$ if  $a_i$ is the only non-zero element in the
       column that contains $a_i$. 
     \item[(3)] $\det P$ factors into a product of a monomial in $a_1, \ldots, a_N$ and binomials of the 
       form $a_{i_1}a_{i_2}\cdots a_{i_r} \pm b_{i_1}b_{i_2}\cdots b_{i_r}$. 
      \item[(4)]If a binomial $a_{i_1}a_{i_2}\cdots a_{i_r} \pm b_{i_1}b_{i_2}\cdots b_{i_r}$ is a factor of
       $\det P$, then we can arrange the order of the variables so that $a_{i_j}$ and $b_{i_j}$
       are in the same row and  $a_{i_{j}}$ and $b_{i_{j-1}}$ are in the same column. (We let $b_{i_0}=b_{i_r}$.)
  \end{enumerate}
\end{proposition}

For proof we use a digraph associated to the matrix $P$ as defined in the next Definition.

\begin{definition} \label{digraph}  
  In the same notation of Proposition~\ref{binomial_matrix_2}, we put $X=A \sqcup B$ and call it the set of vertices.
  We define the set $E$ of arcs to be the union of the two sets  
  \[\{a_i \to b_j | \mbox{$a_i$ and $b_j$ are in the same row} \}   \]
  and
  \[\{b_j \to a_i | \mbox{ $b_j$ and $a_i$ are in the same column} \}.\]
  We may call $(X,E)$ the {\bf digraph associated to $P$}.   
  (Since we have assumed that $i$th row contains $a_i$ and $b_i$, there is an arc $a_i \to b_j$ if and only if $i=j$.)  
\end{definition}

  Suppose $ \{a_{j_1}, a_{j_2} , \ldots, a_{j_r} \} \subset A$  and  
  $\{  b_{j_1}, b_{j_2} , \ldots , b_{j_r} \} \subset B$  are  subsets both consisting of $r$ elements.
  
  We call them a {\bf circuit } of $M$  if  $a_{j_i}$  and  $b_{j_i}$   are contained in one row
  and   $b_{j_i}$   and     $a_{j_{i+1}}$
  are  in one column for all   $i=1,2,\ldots, r$.  (We let $a_{j_{r+1}}=a_{j_{1}}$.)   
  A circuit will be represented as 
  \[a_{j_1} \to b_{j_1} \to a_{j_2} \to b_{j_2} \to  \cdots \to  a_{j_r} \to  b_{j_r} \to a_{j_1}.\]
  If we drop the last term from a circuit, we call it a {\bf chain} of $M$.
  A {\bf chain is maximal} if it cannot be embedded into a circuit or a longer chain. 

\begin{proof}[Proof of Proposition~\ref{binomial_matrix_2}]
  \begin{enumerate}
  \item[(1)] Suppose that  $a_i$  is the only non-zero entry of a column.
    Then obviously $a_i$ divides $\det P$.
  \item[(2)] Recall that $a_i$ and $b_i$ are in the $i$th row of $P$.
    If $a_i$ is the single non-zero element in a column, then
    we may make a column operation to clear  $b_i$.  Hence $\det P$ does not involve $b_i$.
  \item[(3)] Let $(X,E)$ be the digraph introduced in Definition~\ref{digraph}.
    It is easy to see that $(X,E)$ decomposes into the disjoint union
    of circuits and maximal chains.  The first element in a maximal chain is an element of $A$ as a single
    non-zero entry of the column. Indeed any element in $B$ can be  prepended by an element in $A$ and
    in addition, an element in $A$ cannot be prepended by an element of $B$
    if and only if it is a single non-zero entry  of the column.
    If an element $a_i$ appears in a row as a single non-zero element of the column,
    we have $\det P=a_i\det P'$, where $\det P'$ is the matrix  $P$ with the row
    and column deleted that contains  $a_i$.  
    Thus it is enough to treat  $P$ in which all columns have two non-zero elements as well as rows.  
    In this case  $(X,E)$ decomposes as a disjoint union of circuits.  A circuit in $(X,E)$ like   
    \[a^{(1)}\to b^{(1)} \to a^{(2)}\to b^{(2)} \to \cdots  \to a^{(k-1)}\to b^{(k-1)} \to a^{(k)}\to b^{(k)} \to a^{(1)}.\]
    gives us a binomial as a divisor of $\det P$. 
    This completes the proof for (3).
  \item[(4)] This follows immediately from (3). 
  \end{enumerate}
\end{proof}

\begin{example} \normalfont                    
\[\det  \begin{pmatrix}a_1&0&0&b_1 \\ 0&a_2&0&b_2 \\ 0&0&a_3&b_3 \\ 0&0&b_4&a_4
    \end{pmatrix} = a_1a_2(a_3a_4 - b_3b_4).\]
  \end{example}


\section{The quadratic binomial complete intersections} 
\label{new_definition}  \label{def_of_Ms}

In this section we denote by $E=K\langle x_1, x_2, \ldots, x_n  \rangle$ the graded vector space spanned by 
the square-free monomials in the variables $x_1, x_2, \ldots, x_n$  over a field $K$. 
As well as $R_{\la}$, we denote by $E_{\la}=K\langle x_1, x_2, \ldots, x_n   \rangle _{\la}$ 
the homogeneous space spanned by the square-free monomials of degree $\la$ over $K$. 

\begin{definition}\normalfont   \label{notation_of_set_of_monomials}
  For all positive integers $\lambda=1,2,3, \ldots , $ we define the sets of monomials  
  $M_1(\lambda)$, $M_2(\lambda)$, $\ldots$,  $M_n(\lambda)$ of degree $\la$  as follows:
  
\begin{align*}
{} &M_1(\la)  =\{ x_1^{\la _1}x_2^{\la _2}x_3^{\la _3}\cdots x_n ^{\la _n}| \sum _{j=1}^n \la _j = \la  \}, \\
& M_2(\la)  =\{ x_1^{\la _1}x_2^{\la _2}x_3^{\la _3}\cdots x_n ^{\la _n}| \sum _{j=1}^n \la _j= \la, \la _1 < 2  \},  \\ 
&M_3(\la)  =\{ x_1^{\la _1}x_2^{\la _2}x_3^{\la _3}\cdots x_n ^{\la _n}|\sum _{j=1}^n \la _j= \la, \la _1 <2 , \la _2 < 2  \}, \\ 
   & \ \ \ \ \ \ \ \ \ \ \ \  \vspace{10ex}    \vdots \\
&M_n(\la)  =\{ x_1^{\la _1}x_2^{\la _2}x_3^{\la _3} \cdots x_n^{\la _n}| \sum _{j=1}^n \la _j= \la, \la _1 <2 ,  \la _2 < 2, \cdots, \la _{n-1} < 2  \}.
\end{align*}
\end{definition}


\begin{proposition}  \label{basic_property}  \normalfont   
\begin{enumerate}
\item[(1)]
  The span of $M_1(\lambda)$ is  $K[x_1,x_2,\ldots, x_n]_{\la}$. 
\item[(2)]
  $M_1(\la) \supset M_2(\la) \supset \cdots  \supset M_{n-1}(\la) \supset M_{n}(\la)$. 
\item[(3)] For all $\lambda \geq 1$, the set $\bigsqcup _{j=1}^n M_j(\la)\otimes e_j$ is in one-to-one correspondence
  with $M_1(\la +2) \setminus K\langle x_1,x_2,\ldots, x_n \rangle _{\la +2}$.
  ($\{e_1, \ldots, e_n\}$ is a set of indeterminantes used to separate the monomials.)
\item[(4)] For all $\la \geq n-1$, 
  the set $\bigsqcup _{j=1}^n M_j(\la)\otimes e_j$ is in one-to-one correspondence with the set of all
  the monomials  in $R_{\la +2}$. 
\end{enumerate}
\end{proposition}

\begin{proof}  
  \begin{enumerate}
  \item[(1)] and (2) are obvious. 
  \item[(3)]  
    Consider the correspondence $\bigsqcup _{j=1}^n M_j(\la)\otimes e_j \to M_{1}(\la+2)$  defined by
    $M_j(\la)\otimes e_j \ni m \otimes e_j \mapsto m x_j^2$.
    We can make the inverse map as follows.  
    Let $m:=x_1 ^{\la _1}\cdots x_n^{\la _n} \in K[x_1, \ldots, x_n]_{\la +2} \setminus  
    K\langle x_1, \ldots, x_n   \rangle _{\la +2} $. 
    Then some exponent $\la _j$ is greater than $1$. Let $j$ be the smallest index such that $\la _j \geq 2$.
    Then we let $m \mapsto (m/x_j^2) \otimes e_j \in M_j(\la) \otimes e_j$.   
  \item[(4)]
    Since  $\lambda \geq n-1$, we have $\lambda +2 \geq n+1$. For such degrees $\lambda$, the homogeneous part
     of $K\langle x_1, \ldots, x_n \rangle$ does not exist. So (1) and (3) imply (4). 
\end{enumerate}
\end{proof}

\begin{remark} \label{important_remark}
  \begin{enumerate}
  \item[(1)]
    We put $M_1(0)= \cdots = M_n(0)=\{1\}$.
  \item[(2)]
    If $\lambda =1$, then $M_1(\lambda)= \cdots = M_n(\lambda)=\{x_1, x_2, \ldots, x_n\}$.
  \item[(3)]  \label{important_remark_item_2}
    For all $\lambda \geq 1$, $M_{j}(\lambda)x_n \subset M_j(\lambda +1)$ for all $j=1,2, \ldots, n$.
    (This will play a crucial role in  the proof of Proposition~\ref{crucial_prop}.) 
\end{enumerate}
\end{remark}


\begin{theorem} \label{main_theorem} \normalfont    
  Let $R=K[x_1, x_2, \ldots, x_n]$ be the polynomial ring over a field $K$ and
  $A=R/I$  a quadratic binomial complete intersection where  the generators of $I$  are put in a
  normal form. Then the set of square-free monomials are a basis of $A$. 
\end{theorem}

\begin{proof}
  
We fix the generators for $I$  as follows.
  \[f_1=a_1x_1^2+ b_1m_1,\]
  \[f_2=a_2x_2^2+ b_2m_2,\]
  \[   \vdots \]  
  \[f_n=a_n x_n ^2+ b_{n} m_n.\]
 ($m_j$ is a  square-free monomial.)  
  First we assume that $a_1, \ldots, a_n, b_1, \ldots, b_n$ are indeterminates. This means we work over
  $K=\pi(a_1, \ldots, a_n, b_1, \ldots, b_n )$, where $\pi$ is a prime field. 
  
Since the growth of the dimension of  $I_{\la}$ is the same as the monomial ideal $(x_1^2, x_2^2, \ldots, x_n^2)$,
we have 
\[\mu(\gm ^{\la -2} I) = \dim _K I_{\la}= \dim _K R_{\la} - {n \choose \la}.\]
($\mu$  denotes the number of generators of an ideal.)
Put $N:=\mu(\gm ^{\la - 2} I)$. 
We want to  specify  $N$ polynomials in $\gm ^{\la -2}I$ suitable for our purpose. 
For a minimal set of generators for $\gm ^{\la -2} I$ we can choose 
the set of polynomials
\[S:=M_1(\la -2)f_1 \cup  M_2(\la -2)f_2 \cup  \cdots \cup  M_{n-1}(\la -2)f_{n-1} \cup  M_n(\la -2)f_n .\]
Note that these unions are in fact disjoint unions and furthermore these elements are linearly independent.
To see this set $b_1=b_2= \cdots = b_n=0$. 
Then one sees easily that $S$ contains all the monomials in $(x_1^2, \ldots, x_n^2) \cap R_{\lambda}$. These are
$\dim _K R_{\lambda} - {n \choose \la}$ in number.
On the other hand, the number of elements  
$|S|$ can be as large as  this number only if the union is the disjoint union.
If we drop the condition $b_1 = \cdots = b_n=0$, 
the linear independence should be easier to prove.    
We rewrite the set $S$ as
$S=\{g_1, g_2, \ldots, g_N\}$. 
Index the monomials in $R_{\lambda}$ in such a way that the last ${n  \choose \lambda}$ are square-free monomials.
Recall that $\dim _K (I_{\la}) + {n \choose \lambda} = \dim R _{\lambda}$. 
Let $C'=(c_{ij})$ be the matrix consisting of the coefficients of the polynomials in $S$. 
So $C'$ satisfies the following equality:

\[ \begin{pmatrix} g_1 \\ g_2 \\ \vdots \\ g_N \end{pmatrix}
=C'\begin{pmatrix} w_1 \\ w_2 \\ \vdots \\ w_N \\ w_{N+1} \\ \vdots \\ w_{N'}  \end{pmatrix}, \]
where $w_1,  w_2, \ldots , w_{N'}$ are all the monomials in $R_{\la}$ and $N'=N+{n \choose \la}$.  
Let $C$ be the submatrix of $C'$  consisting of rows $1, 2, \ldots, N$  and columns $1, 2, \ldots, N$.  
By Theorem~\ref{from_gelfand's_book}, which we prove in the next section,  we have $\det C \neq 0$.
Define the polynomials $g_1',  g_2', \ldots, g_N'$ as follows:
\[\begin{pmatrix}g_1'\\ g_2'\\ \vdots \\ g_N ' \end{pmatrix}=C^{-1}
  \begin{pmatrix}g_1 \\ g_2 \\  \vdots \\ g_N  \end{pmatrix}=C^{-1}
  C'\begin{pmatrix} w_1\\ w_2 \\ \vdots \\ w_{N'}   \end{pmatrix}.\]

  This matrix notation shows that
\[g_k'-w_k= \fbox{ \mbox{ linear combination of square-free monomials}    } \]
for all $k=1,2, \ldots, N$. 

Since the elements  $g_1', \ldots, g_N'$ are a $K$-basis for  
$I_{\la}$, it follows that  any element of $R_{\lambda}$ can be expressed, mod $I_{\lambda}$, 
as a linear combination of square-free monomials of degree $\lambda$.  
Now the proof is complete for the generic case.  Next we assume $R$ is the polynomial ring
over an arbitrary field $K$. Suppose that $I$ is an ideal of $R$ obtained by substituting elements of
$K$ for the variables $a_i, b_i$.  The ideal $I$ is a complete intersection
if and only if the resultant is non-zero. 

Note that we have established a rewriting rule 
which assigns any monomial $m \in R_{\la}$ mod $I$ to 
a linear combination of square-free monomials in 
$R_{\la}$, for every $\la = 2, \ldots, n+1$. 

For each $\la \geq 2$, we have used the matrices $C'$ 
and $C$.  So we index them as $C'(\la)$ and  $C(\la)$, $\la = 2,3,  \ldots, n+1$. 
Define the matrix $C'(2)$ as the coefficient matrix for $f_1, \ldots, f_n$ and 
$C(2)$ is the first $n \times n$ submatrix of $C'(2)$,  which is automatically the
diagonal matrix with diagonal entries $(a_1, a_2, \ldots, a_n)$.  
By Theorem~\ref{from_gelfand's_book}, if the resultant is non-zero, then it implies all 
$C(3), C(4), \ldots, C(n+1)$ are invertible. Hence the proof is complete. 
\end{proof}

\begin{remark} 
\begin{enumerate}
\item It seems conceivable that any quadratic complete intersection defined by quadrics put in
  a normal form can have the set of square-free monomials as a vector space basis. 
  Theorem~\ref{main_theorem} should be regarded as a case where this can be verified. 
\item If each of the quadrics $f_i$ is a product of linear forms, the elements $f_i$ are in 
  a normal form if we adopt the variables  $x_1, \ldots, x_n$ as
  linear factors of $f_1, \ldots, f_n$ respectively.
  Abedelfatah~\cite{abed_abedelfatah} 
  proved that the Artinian algebra defined by such forms can have the square-free monomials as
  a vector space basis.
\end{enumerate}  
\end{remark}

\begin{remark} \label{additional_remark}
  In Notation~\ref{notation_of_set_of_monomials} we introduced the sets of monomials
  $M_1(\lambda),M_2(\lambda),\ldots, M_n(\lambda)$ for all $\lambda \geq 1$ and  used them
  to define the set $S$  in the proof of Theorem~\ref{main_theorem}.  Suppose that
  \[\sigma=\begin{pmatrix} 1&2&\cdots &n \\ 1'&2'&\cdots & n'   \end{pmatrix}\]
  is a permutation of the indices. If we use the order
  \[1' < 2' < \cdots < n', \]
  for the definition of $M_i(\la)$, instead of the natural order
  \[1 < 2 < \cdots < n, \]
  we have the different sets of monomials
  \[M_1'(\lambda),M_2'(\lambda) , \ldots ,M_n'(\lambda).\]
  The flag of subspaces 
  \[M_1'(\la) \supset M_2'(\la) \supset \cdots \supset M_n'(\la)\]
  is different from
  \[M_1(\la) \supset M_2(\la) \supset \cdots \supset M_n(\la).\]
  In this case we should adopt the set $S$ as
  \[S'=M_1'(\la - 2)f_{1'} \cup M_2'(\la - 2)f_{2'} \cup \cdots \cup M_n'(\lambda-2)f_{n'}.\]
  The consideration of the set $S'$ is important in the definition of the resultant of
  $f_1, \ldots, f_n$ for the binomial complete intersection. 
  See the proof of Theorem~\ref{from_gelfand's_book}(4).
  $M_k'(\lambda)$ should not be confused with $M_{k'}(\lambda)$. 
  It is important that $M_k'(\lambda-2)f_{k'}$ contains the polynomial $x_{k'}^{\lambda -2}f_{k'}$ for all $k$. 
  \end{remark}  

\section{Some remarks on the coefficient matrices of generic complete intersections}

We work with the {\em generic} complete intersection generated by
$f_i=a_ix_i^2 + b_i m_i$, where $m_i$ is a square free monomial of degree 2. 
Recall that we have defined the matrices
\[C'(\lambda) \mbox{ and } C(\lambda)\]
for\[\lambda =2,3, \ldots, n+1.\] 
To define them we used  the subsets $M_i(\lambda) \subset R_{\lambda}$ of monomials
for $i=1,2, \ldots, n$ for all $\lambda = 2,3, \ldots, n+1 $.

Actually it is possible to define these sets for all $\lambda > n+1$, although we do not need them.
From now on $C(\lambda)$ are defined for all $\lambda \geq 2.$ 
\begin{lemma}\label{key_lemma}   \label{lemma_on_C_and_C'}  \normalfont 
  Let $A=\{a_1, \ldots, a_n\}$, $B=\{b_1, \ldots, b_n\}$.  The matrices $C(\lambda)$ and $C'(\lambda)$
  have the following property.  
  \begin{enumerate}  
  \item[(1)] $C(\lambda)$ is a square matrix of size  $\dim _K R_{\lambda} -  {n \choose \lambda}$.
  \item[(2)] Each row of $C'(\lambda)$ contains exactly one element in the set $A$ and  one element in $B$.   
  \item[(3)] Each row of $C(\lambda)$ contains  exactly one element in the set $A$ and at most one element in $B$.    
  \item[(4)] A column of $C(\lambda)$ contains exactly one element in the set $A$. (It  may contain none of
    or many of the elements in $B$.) 
  \item[(5)] For any $i$ and $\lambda \geq 2$, $a_i$ divides $\det C(\lambda)$. 
    
  \end{enumerate}
\end{lemma}

\begin{proof}
    (1) was proved earlier. 
  Recall that a row of the matrix $C'(\lambda)$ is defined as the coefficients of
  $w_kf_i$ for some $i$ with some monomial  $w_k$ of degree $\lambda -2$.
  This proves (2). 
   Recall that the last   ${n \choose \lambda}$  columns of $C'(\lambda)$ are indexed by square-free monomials.  On the
  other hand $a_i$ can appear only in the columns indexed by monomials which are divisible by  $x_i^2$ for some $i$.   
  This proves (3).  Suppose that $a_i$ appears in a column indexed 
  by a monomial $w$ with $w=x_1^{\la _1}x_2^{\la _2} \cdots x_n^{\la _n}$.
  It implies that
  \[\la _1 < 2, \la _2 < 2, \cdots, \la _{i-1} < 2  \leq  \la _i. \]
  Hence $a_k$ cannot appear in this column if $k\neq i$. This proves (4).    
  The column of $C(\lambda)$ indexed by $x_i^{\lambda}$, regarded as a monomial, contains $a_i$ and all  other
  entries are $0$. This proves (5). 
\end{proof}

We  define a circuit in $C(\lambda)$ in the same way as  $P$ considered in Proposition~\ref{binomial_matrix_2}.
Namely, a {\bf circuit} in $C(\lambda)$ is a sequence of elements in $A \sqcup B$
  \[a^{(1)} \to b^{(1)} \to a^{(2)} \to b^{(2)} \to \cdots  \to a^{(r)} \to b^{(r)} \to a^{(1)} \]
  such that $a^{(k)}$ and $b^{(k)}$ are in the same row and
  $b^{(k)}$ and $a^{(k+1)}$ and are in the same column of $C(\lambda)$.
  (The matrix $C(\la)$ differs from $P$ in that the same element occurs in different rows. Thus the repetition may exit in a circuit.)
  It is easy to see that if there is a  circuit in $C(\lambda)$, it gives us a binomial 
\[a_{i_1}^{\al _{i_1}}a_{i_2}^{\al _{i_2}} \cdots a_{i_r}^{\al _{i_r}}  \pm b_{i_1}^{\al _{i_1}}b_{i_2}^{\al _{i_2}} \cdots b _{i_r}^{\al _{i_r}}.\]
  as a factor in the determinant of $C(\la)$. 
We will say that {\bf two circuits are the same} if they give the same determinant.  
Note that a submatrix of $C(\lambda)$ whose determinant is a binomial is another name for a
circuit.  In this sense $C(\lambda)$ and $C(\lambda')$, $\la \neq \la '$, can have the same circuit. 

\begin{proposition} \label{crucial_prop}  \normalfont  
Suppose that a circuit exists in $C(\lambda)$. Then $C(\lambda+1)$ contains the same circuit. 
\end{proposition}

\begin{proof}
  Recall that the columns of $C(\lambda)$ are indexed by certain monomials.  
By the definition of the elements $\{g_1, g_2, \ldots, g_N\}$ to be the set $S$, we may index the 
rows of $C(\la)$ by the elements of $M_j(\lambda - 2)\otimes e_j$. 
  If we multiply the elements (as indices) by $x_n$, they remain as indices for
  $C(\lambda +1)$, since the multiplication by $x_n$ does not change the exponents  
  of monomials except the exponent of $x_n$ itself. (See Remark~\ref{important_remark}(3).) 
    Suppose that
    $\{U_1, U_2, \ldots, U_r\}$ are indices of rows  
    and   $\{W_1, W_2, \ldots, W_r\}$ are
    indices of columns 
    which gives us a circuit of $C(\la)$.  
  Then the submatrix of $C(\lambda +1)$ consisting of rows and columns indexed by
  $\{U_1x_n, U_2x_n, \ldots, U_rx_n\}$ and  $\{W_1x_n, W_2x_n, \ldots, W_rx_n\}$ 
  gives us the same circuit in $C(\lambda +1)$. 
  This proves the assertion.  
\end{proof}

\newcommand{\res}{{\rm Res}}
We denote by $\res(f_1, \ldots, f_n)$ the resultant of $f_1, \ldots, f_n$.
The resultant is a polynomial in the variables of the coefficients $a_1, \ldots, a_n, b_1, \ldots, b_n$, but 
even after the coefficients are substituted for elements in $K$, we call it the resultant. 
The ideal $I$ obtained by substitution is a complete intersection if and only if the resultant
does not vanish. For details see Gelfand et.\ al \cite{GKZ}. 

Define $\Delta _{\lambda}: =\det C(\lambda)$.  
Now we can prove the main theorem of this paper.  

\begin{theorem}  \label{from_gelfand's_book}  \normalfont  
\begin{enumerate}
\item[(0)]  
  $a_1a_2\cdots a_n \mid \Delta _{\lambda}$ for any $\lambda \geq 2$.  
\item[(1)]  
  $a_1a_2\cdots a_n \mid \res(f_1, \ldots, f_n)$. 
\item[(2)]
  $   \res(f_1, \ldots, f_n) \mid   \Delta_{n+1}$.   
\item[(3)]  
  $\sqrt{\Delta_2} \mid \sqrt{\Delta_3} \mid \cdots \mid \sqrt{\Delta_{n+1}}$.   ($\sqrt{\Delta}$ means
  that the  exponents are replaced by $1$ in the factorization of $\Delta$.)
\item[(4)]
  A circuit that appears in some $\Delta(\lambda)$ is a factor of $\res(f_1, \ldots, f_n)$. 
\item[(5)]  $\sqrt{\res(f_1, f_2, \cdots, f_n)}= \sqrt{\Delta_{n+1}}$.  
\end{enumerate}
\end{theorem}

\begin{proof}
  \begin{enumerate}
  \item[(0)]
    This is proved in Lemma~\ref{lemma_on_C_and_C'}(5). 
  \item[(1)]
  It is easy to see that if we set $a_i$ =0 for some $i$, the ideal  $(f_1, \ldots, f_n)$ cannot be a complete intersection. Therefore
  $a_1\cdots a_n$ divides $\res(f_1, \ldots, f_n)$.
\item[(2)]  See \cite[p.429]{GKZ}.
\item[(3)]  It is easy to see that $\Delta _2 =a_1a_2\cdots a_n$. For $\lambda > 2$, it is also easy to see that
  $a_i$ divides $\Delta _{\lambda}$. Suppose that a binomial is a factor of $\Delta _{\lambda}$. Then
  it is a factor of $\Delta _{\lambda+1}$ by Proposition~\ref{crucial_prop}. 
\item[(4)] In \S\ref{def_of_Ms} we constructed the sets  $\{M_j(\lambda) \}$. They depend on the
  order of the indices $1 < 2< \ldots < n$.  Suppose that
  \[\sigma:=\begin{pmatrix} 1 & 2 & \cdots & n \\ 1' & 2' &  \cdots & n' \end{pmatrix}   \] is a permutation of indices.
  Then with the order $1' < 2 ' < \cdots < n'$
  we can obtain another  sequence of determinants: 
  \[\Delta _2 ^{\sigma} , \Delta _3 ^{\sigma},  \ldots, \Delta _{n+1} ^{\sigma}.\]
  It is known that the GCD of $\{\Delta _{n+1}^{\sigma}\}$,  where $\sigma$ runs over
  all cycles of length $n$  
  \[\sigma=\begin{pmatrix} 1 & 2 & \cdots & n-1 & n \\ k & k+1 &  \cdots& k-2 & k-1 \end{pmatrix},    \] 
  gives us the
  resultant $\res(f_1, \ldots, f_n)$.  (See \cite[429]{GKZ}.) Thus to prove the claim
  it suffices to show that if a circuit  appears in
  one of $\Delta _{\lambda}$, with $\lambda \leq n$, then that circuit also appears in $\Delta_{n+1} '$ which is 
  obtained based on a permutation  $\sigma$, whatever the permutation is. 
  This is easy to see since $C'(n+1)=C(n+1)$ up to permutation of rows and columns, every circuit in $C'(\lambda)$
  for $\lambda \leq n$ is contained in $C(n+1)$. 
\item[(5)] Put $r=\res(f_1, \ldots, f_n)$.  By (2) we have  $\sqrt{r} \mid \sqrt{\Delta{n+1}}$. A factor of
  $\Delta_{n+1}$ is either a factor of $a_1\cdots a_n$ or a circuit in $C(n+1)$. We have seen that
  each $a_i$ divides $r$. On the other hand if a circuit divides $C(n+1)$, it divides  $r$ by (4). This completes the proof.    
  \end{enumerate}
\end{proof}

\section{Some examples}
In this section we set $K=\pi(a_1, \ldots, a_5, p_1,\ldots, p_5 )$, the rational function field
over the prime field $\pi$, and $R=K[x_1, \ldots, x_5]$.
By $\sigma=\begin{pmatrix} 1&2&3&4&5 \\ 2&3&4&5&1\end{pmatrix}$, we denote the cyclic permutation of the indices.
If $f \in R$, we denote by $f^{\sigma}$ the polynomial obtained from $f$ by substituting
the indices $i$ by $\sigma(i)$.   
In Example~\ref{ex_1}, the polynomials $f_i$ are determined by $f_1$ by the rule
$f_{i}=f_{i-1}^{\sigma}$ for $i=2,3,4,5$.   

\begin{example}  \label{ex_1} \normalfont
\begin{enumerate} \mbox{}
\item
  If $f_1=a_1x_1^2+ p_1x_1x_2$, we have   
$\res(f_1, \ldots, f_5) = (a_1a_2a_3a_4a_5)^{14}(a_1a_2a_3a_4a_5 + p_1p_2p_3p_4p_5)$ 
\item
  If $f_1=a_1x_1^2+ p_1x_2x_3$,  
$\res(f_1, \ldots, f_5) = (a_1a_2a_3a_4a_5)^{5}(a_1a_2a_3a_4a_5 + p_1p_2p_3p_4p_5)^{11}$
  \item
  If $f_1=a_1x_1^2+ p_1x_2x_5$,  
$\res(f_1, \ldots, f_5) = (a_1a_2a_3a_4a_5)^{11}(a_1a_2a_3a_4a_5 + p_1p_2p_3p_4p_5)^5$
\end{enumerate}
\end{example}

There are $10$ square free monomials in $R_5$. In all cases
the resultant is one of the above three types. It is known that the resultant is a polynomial of
degree $80$. If all $f_j$ factor into two linear forms, the resultant can be computed by
\cite[Chapter~13, Propsotion~1.3]{GKZ}.
This was also computed by Abedelfatah~\cite{abed_abedelfatah} without referring to the resultant. 

In the following table 
$\al, \be$ are the integers such that
\[\res(f_1, \ldots, f_5)=(a_1a_2a_3a_4a_5)^{\al}(a_1a_2a_3a_4a_5+ p_1p_2p_3p_4p_5)^{\be}. \]

\vspace{1ex}

\[\begin{array}{|c|c|c|c|}  \hline
 \mbox{monomial in $f_1$}     & f_1         &  \al   & \be     \\ \hline 
x_1x_2 & a_1x_1^2+ p_1 x_1x_2  &  15  & 1   \\ \hline
x_1x_3 & a_1x_1^2+ p_1 x_1x_3  &  15  & 1   \\ \hline
x_1x_4 & a_1x_1^2+ p_1 x_1x_4  &  15  & 1   \\ \hline
x_1x_5 & a_1x_1^2+ p_1 x_1x_5  &  15  & 1   \\ \hline
x_2x_3 & a_1x_1^2+ p_1 x_2x_3  &  5  & 11   \\ \hline
x_2x_4 & a_1x_1^2+ p_1 x_2x_4  &  5  & 11   \\ \hline
x_2x_5 & a_1x_1^2+ p_1 x_2x_5  &  11 & 5   \\ \hline
x_3x_4 & a_1x_1^2+ p_1 x_3x_4  &  11 & 5   \\ \hline
x_3x_5 & a_1x_1^2+ p_1 x_3x_5  &  5  & 11   \\ \hline
x_4x_5 & a_1x_1^2+ p_1 x_4x_5  &  5  & 11   \\ \hline
\end{array}
\]

\begin{example}  \normalfont
  In this example we chose the square-free monomials rather randomly.
  
  Put
  \[f_1=a_1x_1 ^2 + p_1x_2x_3,\]
  \[f_2=a_2x_2 ^2 + p_2x_3x_5,\]
  \[f_3=a_3x_3 ^2 + p_3x_4x_5,\]
  \[f_4=a_4x_4 ^2 + p_4x_1x_3,\]
  \[f_5=a_5x_5 ^2 + p_5x_1x_2.\]
  Then
  \[\res(f_1,\ldots, f_5)=a_1^9a_2^8a_3^6a_4^{11}a_5^7(a_1^7a_2^8a_3^{10}a_4^5a_5^9+p_1^7p_2^8p_3^{10}p_4^5p_5^9).\]

\end{example}

\section{Relevance to the 2nd Hessian}

\begin{example} \label{ex_3} \normalfont
  Let $K$ and $R$ be the same as in the previous section. We use the notation $v=x_1, w=x_2, \ldots, z=x_5 $ interchangeably.  
  We consider the polynomial \[G=120 vwxyz+s_1+s_2+s_3+s_4+s_5,\]
  where
\begin{align*}
  {}&s_1=- (p_1^3p_3p_4v^5 + p_2^3p_4p_5w^5  + p_1p_3^3p_5x^5 +  p_1p_2p_4^3y^5  +  p_2p_3p_5^3z^5),   \\
  {}&s_2= -20(p_1v^3wz + p_2vw^3x  +  p_3wx^3y + p_4xy^3z  + p_5vyz^3),   \\
  {}&s_3= 20(p_1^2p_3p_4v^3xy + p_2^2p_4p_5w^3yz  +  p_1p_3^2p_5vx^3z + p_1p_2p_4^2vwy^3  + p_2p_3p_5^2wxz^3),\\
  {}&s_4= 30(p_1p_3v^2wx^2 + p_2p_4w^2xy^2 + p_3p_5x^2yz^2 + p_1p_4v^2y^2z + p_2p_5vw^2z^2),  \\
  {}&s_5 =-30(p_1p_2p_4v^2w^2y + p_2p_3p_5w^2x^2z + p_1p_3p_4vx^2y^2 + p_2p_4p_5wy^2z^2 + p_1p_3p_5v^2xz^2).
\end{align*}

  We consider $G$ as a polynomial in the polynomial ring $R=K[v,w,x,y,z]$. $G$ was obtained as the
  Macaulay dual generator of the complete intersection $I=(f_1,f_2,f_3,f_4,f_5)$, where 
  \[f_1=v^2 + p_1wz,\]  
  \[f_2=w^2 + p_2xv,\]
  \[f_3=x^2 + p_3yw,\]
  \[f_4=y^2 + p_4zx,\]
  \[f_5=z^2 + p_5vw.\]  

  As we said in the introduction, the resultant of 
  these elements  is \[(p_1p_2p_3p_4p_5+1)^5.\]  (The polynomial $G$ was computed by the computer
  algebra system Mathematica~\cite{mathematica}.) 

  It is not difficult to verify that $\Ann _R(F) \supset (f_1, \ldots, f_5)$.  
  If $p_1p_2p_3p_4p_5 + 1 \neq 0$, then since $\Ann _R G$ is a
  Gorenstein ideal containing $f_1, \ldots, f_5$, 
  it follows that they coinside: $\Ann _R(G)=(f_1, \ldots , f_5)$ and 
  $A:=K[v,w,x,y,z]/\Ann _R(G)$ has the Hilbert function $(1\ 5 \ 10 \ 10  \ 5 \ 1)$.

  If $p_1p_2p_3p_4p_5+1 =  0$, then we can calculate that the algebra  
  $A=K[v,w,x,y,z]/\Ann _R(G)$ has the Hilbert function $(1\ 5 \ 5 \ 5  \ 5 \ 1)$.

  Since we know that the square free monomials are linearly independent,
  the second Hessian matrix of $G$ is, in this case, computed as the $10 \times 10$ matrix 
  \[H^2(G)=\left( \frac{\pa^4 (G)}{\pa x_i \pa x_j \pa _k \pa _l } \right) _{(1 \leq i < j \leq 5),(1 \leq k  < l \leq 5)} .\]
  For details of higher Hessians, see \cite{maeno_watanabe}.    
\newcommand{\hess}{{\rm hess}}
  Let $\hess ^2(G)$ be the determinant of $H^2(G)$.  
  It is a polynomial in $v, \ldots, z, p_1, \ldots, p_5$.  We may regard $\hess ^2$ as a polynomial in
  $v, w, x, y, z$ with coefficients in $\pi[p_1, p_2, p_3, p_4,p_5]$.  Let $P$ be the ideal in the
  polynomial ring $\pi[p_1, p_2, p_3, p_4, p_5]$ generated by the coefficients of $\hess ^2$,
  where $\pi$ is a prime field.  
  It has many complicated generators but surprisingly enough, 
  it turns out that the ideal
  $P$ is a principal ideal generated by  of $(1+p_1p_2p_3p_4p_5)^5$. The computation was done also by Mathematica~\cite{mathematica}.  
\end{example}

\begin{example}   \label{ex_4}  \normalfont  
  Let $F$ be the polynomial in the first paragraph of Introduction.
$F$ was in fact obtained as the Macaulay dual generator fo the complete intersection $I=(f_1, \ldots, f_5)$, where 
  \[f_1=v^2 + p_1wx,\]
  \[f_2=w^2 + p_2xy,\]
  \[f_3=x^2 + p_3yz,\]
  \[f_4=y^2 + p_4vz,\]
  \[f_5=z^2 + p_5vw.\]
  As in the previous example, let $H^2(F)$  be the second  Hessian of $F$. Then the coefficient ideal
  in $K[p_1, \ldots, p_5]$ turns out to be the unit ideal.  
  Hence
  the algebra $\pi[p_1, \ldots, p_5][v,w,x,y,z]/I$ gives a flat family of Artinian Gorenstein algebras over
  $K=\pi[p_1, \ldots, p_5]$.  
  The fiber is a complete intersection if and only if  $p_1p_2p_3p_4p_5+1 \neq 0$, and otherwise 
  it is a Gorenstein algebra defined by a $7$-generated ideal. (This is a computational result.) 
  \end{example}

\section*{Acknowledgement}
The third author would like to thank K.\ Yanagawa
for a helpful conversation for 
quadratic binomial complete intersections.


\end{document}